\documentclass{amsart}
\pdfoutput=1
\usepackage{amsthm}
\usepackage{amsmath}

\usepackage[pdfusetitle]{hyperref}

\usepackage{tikz}
\usepgflibrary{shapes.misc}
\usetikzlibrary{arrows,positioning,fit,backgrounds}
\usetikzlibrary{shapes.misc} 
\usetikzlibrary{calc}

\DeclareMathOperator{\pd}{PD}
\DeclareMathOperator{\gl}{GL}
\DeclareMathOperator{\spl}{SL}

\newtheorem{lemma}{Lemma}

\begin{document}

\title{A Note on Affine Invariant Cost Functions}

\author{Jingbo Liu}
\address{Institute for Data, Systems, and Society (IDSS), Massachusetts Institute for Technology}
\email{jingbo@mit.edu}

\begin{abstract}
We show that any affine invariant function on the set of positive definite matrices must factor through the determinant function, as long as the restriction of the function to scalar matrices is surjective.
A motivation from robust statistics is discussed.
\end{abstract}

\date{\today}

\maketitle

\subsection*{Notation} 
We use the following notations for sets of matrices.
\begin{itemize}
\item $\pd(n)$: the set of $n\times n$ positive definite matrices.
\item $\gl(n)$: the general linear group, i.e., the set of invertible $n\times n$ matrices.
\item $\spl(n)$: the special linear group, i.e., the set of $n\times n$ matrices whose determinant equals $1$, which is a subgroup of $\gl(n)$.
\item $I$ denotes the identity matrix.
\end{itemize}
Unless otherwise specified, all matrices are over the reals. We also use a few standard group-theoretic notations: suppose that $\mathcal{G}$ is a group and $\mathcal{K}$ is a normal subgroup, then the \emph{quotient group} $\mathcal{G}/\mathcal{K}$ is the collection of cosets of $\mathcal{K}$, equipped with the group operations inherited from $\mathcal{G}$.
The \emph{canonical map} is the homomorphism $\phi\colon \mathcal{G}\to\mathcal{G}/\mathcal{K}$ sending an element $g\in\mathcal{G}$ to $g\mathcal{K}$, the $\mathcal{K}$-coset containing $g$, and $\mathcal{K}$ is the kernel of this homomorphism,
denoted by $\ker(\phi)$.

\subsection*{Definition} Let $f\colon \pd(n)\to \mathcal{S}$ be a map from the set of $n\times n$ positive semidefinite matrices to an arbitrary set $\mathcal{S}$.
$f$ is said to be \emph{affine invariant}, if
\begin{align}
f(M)=f(N) \quad \Longrightarrow \quad f(A^{\top}MA) = f(A^{\top}NA) 
\label{e_inv}
\end{align}
for any $M, N\in \pd(n)$ and $A \in \gl (n)$.

\subsection*{Theorem} Let $f$ be a surjective affine invariant map between $\pd(n)$ and $\mathcal{S}$,
and suppose that 
\begin{align}
\cup_{s\in(0,\infty)} \{f(sI)\}=\mathcal{S}.
\label{e_surj}
\end{align}
Then $f$ factors as 
\begin{align}
f = b\circ H \circ\det
\end{align}
where:
\begin{itemize}
\item $\det$ denotes the determinant;
\item $H$ is a homomorphism from $(0,\infty)$, where $(0,\infty)$ is viewed as the multiplicative group);
\item $b$ is a bijection from the image of $H$ to $\mathcal{S}$. 
\end{itemize}
The homomorphism $H$ can be viewed as the canonical map $(0,\infty)\to (0,\infty)/\ker(H)$ (up to isomorphism). 
Hence $H$ is uniquely specified by $\ker(H)$,
a subgroup of the multiplicative group $(0,\infty)$,
which corresponds to a subgroup of the additive group $\mathbb{R}$ by taking the logarithm.
Examples of subgroups of $\mathbb{R}$ include the lattice $a\mathbb{Z}$ ($a\in\mathbb{R}$) and the rationals $\mathbb{Q}$.
When we take $a\mathbb{Z}$ and $a=0$ we have 
\begin{align}
f(M)=\det(M)
\end{align}
since $\ker(H)=\{0\}$ and $H$ is the identity map.
When $a\mathbb{Z}$ but $a\neq 0$, we can write
\begin{align}
f(M)=2^{ak}\det(M)
\end{align}
where $k$ is the unique integer such that $2^{ak}\det(M)\in[1,2^a)$.

\subsection*{Remark} The assumption \eqref{e_surj} is necessary: 
for example, the identity map $f\colon M\mapsto M$ is affine invariant, but fails \eqref{e_surj}, 
and does not factor through $\det$.

We establish several lemmas before the proof of the main theorem,
and at the end of the note a motivation from robust statistics is discussed.
\begin{lemma}\label{lem1}
\begin{align}
\mathcal{K}:= \{A\in \gl(n)\colon  f(A^{\top}A) =f(I) \}
\end{align}
forms a normal subgroup of $\gl(n)$.
\end{lemma}
\begin{proof}
We first show that $\mathcal{K}$ is a subgroup. 
Indeed, for any $A, B \in \gl(n)$ such that $f(A^{\top}A) = f(B^{\top}B) = f(I)$, we can deduce from \eqref{e_inv} that $f(B^{\top}A^{\top}AB) = f(I)$, and $f(A^{-\top}A^{-1}) = f(I)$,
meaning that $AB, A^{-1}\in\mathcal{K}$.

Next we show that $\mathcal{K}$ is normal.
First note the following property: if $Q$ is an orthogonal matrix and $A \in \gl(n)$, then 
\begin{align}
f(A^{\top}A) = f(Q^{\top}A^{\top}AQ).
\label{e_orthogonal}
\end{align} 
To see this, assume that $f(A^{\top}A) = f(sI)$ for some $s\in(0,\infty)$, then apply the affine invariance assumption \eqref{e_inv}.

To show the normality of $\mathcal{K}$ it suffices to prove that 
\begin{align}
B^{\top}A^{\top}B^{-1}A^{-1}\in\mathcal{K},
\quad \forall A, B \in \gl(n),
\label{e_com}
\end{align}
which, by the affine invariance assumption, is equivalent to 
\begin{align}
f(A^{\top}B^{\top}BA) = f(B^{\top}A^{\top}AB).
\label{e_commute}
\end{align}
We now prove \eqref{e_commute} by singular value decomposition.
Let us assume that $A=P_1\Lambda_1Q_1$ and $B=P_2\Lambda_2Q_2$ where $P_i,Q_i$ are orthogonal matrices, and $\Lambda_i$ are diagonal matrices, $i=1,2$. Then we have
\begin{align}
f(A^{\top}B^{\top}BA)
&= f((P_2\Lambda_2Q_2P_1\Lambda_1Q_1)^{\top}(P_2\Lambda_2Q_2P_1\Lambda_1Q_1))
\\
&= f((\Lambda_2Q_2P_1\Lambda_1)^{\top}\Lambda_2Q_2P_1\Lambda_1)
\label{e7}
\\
&= f(( \Lambda_2\Lambda_1)^{\top}\Lambda_2\Lambda_1)   
\label{e8}
\end{align}
where \eqref{e7} follows from the property in \eqref{e_orthogonal}.
\eqref{e8} follows from affine invariance and the property in \eqref{e_orthogonal}.
Similarly we also have $f(B^{\top}A^{\top}AB) = f(( \Lambda_2\Lambda_1)^{\top}\Lambda_2\Lambda_1)$, therefore \eqref{e_commute} is established,
and $\mathcal{K}$ must be normal.
\end{proof}

\begin{lemma}
There is a bijection $b\colon \gl(n)/\mathcal{K}\to \mathcal{S}$, such that the map $A\mapsto f(A^{\top}A)$ factors as the canonical map $\gl(n) \to \gl(n)/\mathcal{K}$ followed by $b$.
\end{lemma}
\begin{figure}[h!]
  \centering
\begin{tikzpicture}
[scale=1]
  \tikzset{
mystyle/.style={
circle,
  inner sep=0pt,
  text width=10mm,
  align=center,
  fill=white
  }
}
  \node[mystyle] (c) {};
  \node[mystyle] (L) [left=2 of c] {$\gl(n)$};
  \node[mystyle] (R) [right=2 of c] {$\gl(n)/\mathcal{K}$}; 
  \node[mystyle] (U) [above=1 of c] {$(0,\infty)$};  
  \node[mystyle] (D) [below=1 of c] {$\mathcal{S}$};   

\draw[->] (L) to node[midway,above,sloped]{$A\mapsto \det^2(A)$}(U);
\draw[->] (U) to node[midway,above,sloped]{$H$}(R);
\draw[->] (L) to (R);
\draw[->] (L) to node[midway,below,sloped]{$ A\mapsto f(A^{\top}A)$}(D);
\draw[->] (R) to node[midway,below,sloped]{$b\colon A\mathcal{K}\mapsto f(A^{\top}A)$}(D);
\end{tikzpicture}
\caption{}
\label{fig1}
\end{figure}
\begin{proof}
We choose $b\colon A\mathcal{K}\mapsto f(A^{\top}A)$ (see Figure~\ref{fig1}), and we need to check that it is well-defined and bijective.

Suppose that $B_1,B_2$ are elements in any given coset $A\mathcal{K}$. 
Then, $B_1B_2^{-1}\in\mathcal{K}$, and hence $f((B_1B_2^{-1})^{\top}B_1B_2^{-1})=f(I)$, which by the affine invariance assumption is equivalent to $f(B_1^{\top}B_1)=f(B_2^{\top}B_2)$.
This means that the choice of representative element in $A\mathcal{K}$ is immaterial, and hence $b$ is well-defined.

Injectivity follows since for any $B_1,B_2\in\gl(n)$, $f(B_1^{\top}B_1)=f(B_2^{\top}B_2)$ implies that $B_1$ and $B_2$ are in the same $\mathcal{K}$-coset. 
Surjectivity follows since $A \mapsto A^{\top}A$ is onto $\pd(n)$.
\end{proof}

\begin{lemma}
The canonical map $\gl(n)\to \gl(n)/\mathcal{K}$ factors as $A \mapsto \det^2(A)$ followed by a homomorphism $H$ from $(0,\infty)$ (viewed as a multiplicative group) to some quotient group.
\end{lemma}
\begin{proof}
In \eqref{e_com} we have proved that $\mathcal{K}$ contains the subgroup generated by the commutators, called the \emph{derived subgroup}.
We now invoke the following known fact (see \cite{derived}):
\begin{quote}
\it
The derived subgroup of the general linear group (the group of invertible $n\times n$ matrices over a given field) is the special linear group (the group of $n\times n$ matrices of determinant $1$), under either of the follows conditions
\begin{itemize}
\item $n\ge 3$.
\item The field has at least three elements.
\end{itemize}
\end{quote}
Let us remark that this fact can be shown by checking that under either of the above conditions, the \emph{elementary matrices}\footnote{$A\in\gl(n)$ is said to be an elementary matrix if $A-I$ has at most one non-zero off-diagonal element. The set of elementary matrices generate the $\spl(n)$.} can be expressed as a commutator (see \cite{commutator}).
Now in our problem the underlying field $\mathbb{R}$ fulfills the second condition above, therefore the derived subgroup is $\spl(n)$ and hence $\mathcal{K}$ contains $\spl(n)$.
Further since $f(A^{\top}A)=f((A^{\top}A)^{1/2})=f(I)$ for any $A$ with $\det(A)= -1$, we see that $\mathcal{K}$ actually contains the larger subgroup $\{A\colon \det^2(A) = 1\}$.
Therefore by the \emph{fundamental theorem on group homomorphisms}, 
we see that the canonical map $\gl(n)\to \gl(n)/\mathcal{K}$ must factor as claimed in the lemma, and that $(0,\infty)/\ker(H)$ is isomorphic to $\gl(n)/\mathcal{K}$.
\end{proof}

\begin{proof}[Proof of the main theorem]
Let $M\in\pd(n)$ be arbitrary; we have (see Figure~\ref{fig1})
\begin{align}
f(M) &= f((M^{1/2})^{\top} M^{1/2})
\\
&=b (H ({\rm det}^2(M^{1/2})))
\\
&=b(H(\det (M))).
\end{align}
\end{proof}

\subsection*{Motivation from robust statistics} 
In the robust mean estimation problem, a statistician is given a set $\{X_1,\dots,X_k\}$ of data points in $\mathbb{R}^n$, 
some of which have been corrupted,
and the goal is to estimate the mean of the uncorrupted data points.
The classical \emph{minimum covariance determinant estimator} (see \cite[(3.5)]{rousseeuw1985multivariate})
returns the following estimate:
\begin{align}
T(X_1,\dots,X_k):=&\textrm{mean of the $h$ points of $\{X_1,\dots,X_k\}$ for which the determinant}
\nonumber\\
&\textrm{of the covariance matrix is minimal} 
\label{e_mcd}
\end{align}
where $h\le k$ is some integer,  usually chosen based on the proportion of the corrupted data points.

While other objective functions could be used in lieu of the determinant function in \eqref{e_mcd},
it has been well-noted that the determinant function enjoys the desirable equivariant property, that is, the estimator commutes with affine transformations
(see e.g.\ \cite{rousseeuw1985multivariate}
\cite{donoho1992breakdown}).
The equivariant property clearly follows from the affine invariant property of the determinant function.
In contrast, the main theorem in this note establishes the converse statement, that any affine invariant function must factor through the determinant function.

\section*{Acknowledgement}
The author started looking into the problem discussed in this note after a discussion with Prof.\ Jiantao Jiao on robust statistics.
The author gratefully acknowledges Prof.\ Jiantao Jiao for his comments and for encouraging the author to write up this note.

\bibliographystyle{plainurl}
\bibliography{ref_aff}

\begin{thebibliography}{1}

\bibitem{derived}
Derived subgroup of general linear group is special linear group, Revision as
  of 20:46, 10 January 2010.
\newblock URL:
  \url{https://groupprops.subwiki.org/w/index.php?title=Derived_subgroup_of_general_linear_group_is_special_linear_group&oldid=23170}.

\bibitem{commutator}
Every elementary matrix of the first kind is a commutator of elementary
  matrices of the first kind, Revision as of 21:18, 5 April 2009.
\newblock URL:
  \url{https://groupprops.subwiki.org/w/index.php?title=Every_elementary_matrix_of_the_first_kind_is_a_commutator_of_elementary_matrices_of_the_first_kind&oldid=17627}.

\bibitem{donoho1992breakdown}
David~L Donoho and Miriam Gasko.
\newblock Breakdown properties of location estimates based on halfspace depth
  and projected outlyingness.
\newblock {\em The Annals of Statistics}, 20(4):1803--1827, 1992.

\bibitem{rousseeuw1985multivariate}
Peter~J Rousseeuw.
\newblock Multivariate estimation with high breakdown point.
\newblock {\em Mathematical statistics and applications}, 8(283-297):37, 1985.

\end{thebibliography}

\end{document}